\documentclass[11pt]{article}
\usepackage{graphicx}
\usepackage{amsmath}

\usepackage{amsthm}
\newfont{\bb}{msbm10 at 11pt}
\def\r{\hbox{\bb R}}

\newtheorem{theorem}{Theorem}[section]
\newtheorem{corollary}[theorem]{Corollary}

\newtheorem{lemma}[theorem]{Lemma}
\newtheorem{remark}[theorem]{Remark}

\begin{document}

\title{Constant mean curvature surfaces in Sol\\ with non-empty boundary}
\author{Rafael L\'opez\footnote{Partially supported by MEC-FEDER
 grant no. MTM2007-61775 and
Junta de Andaluc\'{\i}a grant no. P06-FQM-01642.}\\
Departamento de Geometr\'{\i}a y Topolog\'{\i}a\\ Universidad de Granada\\ 18071 Granada, Spain\\
email: rcamino@ugr.es}

\date{}
\maketitle

\begin{abstract} In homogenous space Sol we study compact surfaces  with constant mean curvature and with non-empty boundary. We ask how the geometry of the boundary curve  imposes restrictions over all possible configurations
that the surface can adopt.  We obtain a flux formula  and we establish  results that assert that, under some restrictions, the symmetry of the boundary is inherited  into  the surface.
\end{abstract}



%%%%%%%%%%%%%%%%%%%%%%%%%%%%%%%%%%%%%%%%%%%%%%%%%%%%%%%%%%%%%%%
\section{Introduction}
%%%%%%%%%%%%%%%%%%%%%%%%%%%%%%%%%%%%%%%%%%%%%%%%%%%%%%%%%%%%%%%

The space Sol is a  simply connected  homogeneous 3-manifold whose isometry group has dimension $3$
and it belongs to one of the eight models of geometry of Thurston \cite{th}.  As a Riemannian manifold,
the space Sol  can be
represented by $\r^3$ equipped with the metric
$$\langle,\rangle=ds^2=e^{2z}dx^2+e^{-2z}dy^2+dz^2$$
 where $(x,y,z)$ are canonical coordinates of $\r^3$. The space Sol, with the group operation
 $$(x,y,z)\ast (x',y',z')=(x+e^{-z}x',y+e^{z}y',z+z'),$$
is a Lie group and the metric $ds^2$ is  left-invariant.

Although Sol is a homogenous space and the action of the isometry group is transitive, the fact that the number of isometries is low (for example, there are no rotations) makes that the knowledge of the geometry of submanifolds is far to be complete. For example, it is known the geodesics of space (\cite{tr}) and more recently, the totally umbilical surfaces (\cite{sot}), some properties on  surfaces with constant mean curvature (\cite{dm,il}) and invariant surfaces with constant curvature (\cite{lo}).

In this paper we consider compact surfaces with constant mean curvature (CMC) in Sol. First, we establish the following definition. Let $\Gamma$ be a closed curve of Sol and let $x:M\rightarrow Sol$ be an isometric immersion of a
compact surface $M$ with  non-empty boundary $\partial M$. We say that \emph{$\Gamma$ is the boundary of $x$}  if $x:\partial M\rightarrow\Gamma$ is a diffeomorphism. We also say that $\Gamma$ is the boundary of $M$ if we
the immersion is known.

We ask how the boundary of the surface has  influence on the geometry of the whole of the surface. The relationship between the geometry of a surface and the geometry of its boundary has been asked in other ambient spaces, specially, in  Euclidean space. A  natural question is if a CMC surface inherits the symmetries of its boundary. To be precise,  let $\Gamma$ be a closed curve and $\Phi$ an isometry of the ambient space  Sol such that $\Phi(\Gamma)=\Gamma$. If $M$ is a compact CMC surface
 with boundary $\Gamma$, do hold  $\Phi(M)=M$?

On the other hand, we ask  if there exists some type of restrictions for a CMC surface to be the boundary of  a given curve.  Our question formulates as follows: given a closed curve $\Gamma$ in Sol and $H\in\r$,
do any restrictions exist  of the possible values $H$ of the mean curvature of a compact surface in Sol bounded by $\Gamma$?

The space Sol is the space with less isometries among all homogenous spaces and as  a consequence  the hypothesis in our statements will be more restrictive. This will  reflect that the results obtained here are less outstanding than other ambient spaces. As we have pointed out, the isometry group Iso(Sol) has dimension $3$ and the component of the identity is generated by the following families of isometries:
\begin{eqnarray*}
& & T_{1,c}(x,y,z):=(x+c,y,z)\\
& &T_{2,c}(x,y,z):=(x,y+c,z)\\
& &T_{3,c}(x,y,z):=(e^{-c}x,e^{c}y,z+c),
\end{eqnarray*}
where $c\in\r$ is a real parameter. These isometries are left multiplications by elements of Sol and so, they are left-translations with respect to the structure of Lie group. Remark that the elements $T_{1,c}$ and $T_{2,c}$ are Euclidean translations along horizontal vector. Therefore,
Euclidean reflections in the $(x,y,z)$ coordinates with respect to a plane $P_t$ or $Q_t$ are isometries of Sol.
In the  problems posed in this work, by 'symmetry' we mean  invariant by one of the above two kinds of reflections.

The understanding of the geometry of Sol is given by the next three foliations:
\begin{eqnarray*}
& &{\cal F}_1: \{P_t=\{(t,y,z);y,z\in\r\}\}_{t\in\r}\\
& &{\cal F}_2: \{Q_t=\{(x,t,z);x,z\in\r\}\}_{t\in\r}\\
& &{\cal F}_3: \{R_t=\{(x,y,t);x,y\in\r\}\}_{t\in\r}.
\end{eqnarray*}
The foliations ${\cal F}_1$ and ${\cal F}_2$  are determined by the isometry groups $\{T_{1,c}\}_{c\in\r}$ and $\{T_{2,c}\}_{c\in\r}$ respectively, and they
describe (the only) totally geodesic surfaces of Sol, being  each leaf of the foliation isometric to a hyperbolic plane;  the foliation ${\cal F}_3$ realizes by minimal surfaces, all them isometric to  Euclidean plane.

In this work   the boundaries of our surfaces  lie in a totally geodesic surface or in a leaf of ${\cal F}_3$. After an isometry of the ambient space, a leaf of ${\cal F}_2$ can be carried into a leaf of ${\cal F}_1$
by using the isometry of Sol given by $\phi(x,y,z)=(y,x,-z)$, and hence one may assume that it is $P_0$ by a horizontal translation.
Similarly, any leaf $R_t$ of ${\cal F}_3$ carries to $R_0$ by taking the isometry $\phi(x,y,z)=(e^{t}x,e^{-t}y,z-t)$.
As a consequence, in this paper we consider that the boundary of the surface lies in the plane
$P:=P_0$ or $R:=R_0$.

In section \ref{s-flux} we establish the flux formula for compact CMC surfaces in Sol and we obtain upper bounds for the possible (constant) values of mean curvatures that a surface can take to be boundary of a given curve. In this sense,
we show we show (Corollary \ref{c-cir} and Theorem \ref{t-ere})
\begin{quote}{\it If $\Gamma$ is a circle of curvature $c$ contained in a totally geodesic plane and $M$ is a surface spanning $\Gamma$ with constant mean curvature $H$, then
$$|H|\leq \frac{c+1}{c-1}\frac{c+\sqrt{c^2-1}}{2}.$$
If $\Gamma$ is a circle of curvature $c$ contained in the plane $z=0$, then $|H|\leq \sqrt{c^2+1}$.}
\end{quote}
In section
\ref{s-emb} we apply the Alexandrov reflection method to obtain results that assures that  a CMC embedded surface inherits the symmetries of its boundary, as well as the possible configurations that a such surface can adopt. Finally, in section \ref{s-app} we combine the flux formula with the maximum principle in order to establish results of symmetry and uniqueness.

%%%%%%%%%%%%%%%%%%%%%%%%%%%%%%%%%%%%%%%%%%%%%%%%%%%%%%%%%%%%%%%%%%
\section{The flux formula in Sol}\label{s-flux}
%%%%%%%%%%%%%%%%%%%%%%%%%%%%%%%%%%%%%%%%%%%%%%%%%%%%%%%%%%%%%%%%%%

In this section we deduce a flux formula for CMC surfaces in Sol, which appears usually in the literature of surfaces with constant mean curvature. First, we  recall the Killing vectors fields in Sol (see details in  \cite{dm,tr}; see also the Appendix). With respect to the metric $ds^2$, an orthonormal basis of left-invariant vector fields is given by
$$E_1=e^{-z}\frac{\partial}{\partial x},\hspace*{.5cm}  E_2=e^{z}\frac{\partial}{\partial y},\hspace*{.5cm}  E_3=\frac{\partial}{\partial z}.$$
In the space Sol a basis of Killing vector fields is
$$\frac{\partial}{\partial x},\hspace*{.5cm} \frac{\partial}{\partial y},\hspace*{.5cm}  -x\frac{\partial}{\partial x}+y\frac{\partial}{\partial y}+\frac{\partial}{\partial z}.$$

We now establish the flux formula in Sol following the same steps than in Euclidean space (see originally in \cite{kks}).  Consider $M$ and $D$ two compact
surfaces immersed in Sol with $\partial M=\partial D$ such that
$M\cup D$ is
 an oriented cycle. Let $W$ be the oriented immersed domain in Sol bounded by $M\cup D$.
 Let $N$ and $\eta$ be the unit normal fields to $M$ and $D$ respectively that point inside $W$.
 If $X$ is a Killing vector field in Sol, and because the divergence of $X$ on $W$ is zero, the
 Divergence theorem asserts
\begin{equation}\label{eq0}
\int_M \langle N,X\rangle+\int_D\langle \eta,X\rangle=0,
\end{equation}
Assume now that $M$ is a compact surface of constant mean curvature $H$. Since $X$ is a Killing vector field, the first variation of area is zero. Consequently, we have
\begin{eqnarray*}
0&=&\int_M\mbox{div}_M(X)=\int_M\mbox{div}_M(X^\top)+\int_M\mbox{div}_M(X^\bot)=\\
&=&-\int_{\partial M}\langle \nu,X\rangle-2H\int_M \langle N,X\rangle,
\end{eqnarray*}
where $X^\top$ and $X^\bot$ denote the tangent and the normal components of $X$ with respect to $M$ and $\nu$ is the unit conormal to $M$ along $\partial M$ pointing inside $M$. By using (\ref{eq0}),  we have proved

\begin{lemma}[Flux formula]\label{l-flux} Let $X$ be a Killing vector field of Sol. Consider $M$ an immersed compact CMC surface in Sol and let $D$ be a compact surface  such that $\partial M=\partial D$ and $M\cup D$ is an oriented cycle. Then
\begin{equation}\label{eq3}
\int_{\partial M}\langle \nu,X\rangle=2H\int_D\langle \eta,X\rangle.
\end{equation}
\end{lemma}

Now we are going to put in the flux formula (\ref{eq3}) the Killing vector fields of Sol.

\begin{theorem}\label{t-flux} Let $\Gamma$ be a closed embedded curve included in the plane $P$. If $M$ is a
compact surface spanning $\Gamma$ and with constant mean curvature $H$, then
\begin{equation}\label{flux1}
|H|\leq \exp{\Big(\max_{p\in\Gamma}z(p)-\min_{p\in\Gamma}z(p)\Big)}\frac{L(\Gamma)}{2\ A(D)},
\end{equation}
where $D\subset P$ is the bounded domain by $\Gamma$ and $L(\Gamma)$ and $A(D)$ denote the length of $\Gamma$ and the area of $D$, respectively. Moreover, equality in (\ref{flux1})  holds if and only if
$\Gamma$ is a line of curvature of $M$ and $M$ is orthogonal to $P$ along $\partial M$.
\end{theorem}

\begin{proof} Take $X=\frac{\partial}{\partial x}$ in (\ref{eq3}). First, $\eta=\pm e^{-z}\frac{\partial}{\partial x}=E_1$. Both sides in (\ref{eq3}) yields
$$\int_{\partial M}\langle\nu,\frac{\partial}{\partial x}\rangle\leq \int_{\partial M}\Big|\frac{\partial}{\partial x}\Big|=
\int_{\partial M}e^{z}\leq \exp{(\max_{p\in\Gamma}z(p))}L(\Gamma).$$
$$\Big|2H\int_D\langle\eta,\frac{\partial}{\partial x}\rangle\Big|= 2|H|\int_D e^{z}\geq
2|H| \exp{(\min_{p\in\Gamma}z(p))} A(D).$$
This shows (\ref{flux1}).

In the case that the equality holds in (\ref{flux1}), one concludes  that
$\nu$ is proportional to $\frac{\partial}{\partial x}$, and then $\nu=E_1$. As $\langle N,\frac{\partial}{\partial x}\rangle=0$ along $\partial M$, we have ${\alpha'}\langle N,\frac{\partial}{\partial x}\rangle=0$ on $\partial M$, where $\alpha'$ is a unit tangent vector to $\partial M$. Then
$$\langle\nabla_{\alpha'}N,\frac{\partial}{\partial x}\rangle+\langle N,\nabla_{\alpha'}\frac{\partial}{\partial x}\rangle=0\hspace*{.5cm}\mbox{along $\partial M$}.$$
Parametrize $\Gamma$ as $\alpha(s)=(0,y(s),z(s))$. Using (\ref{cov}), we know
$$\nabla_{\alpha'}\frac{\partial}{\partial x}=z'e^z E_1\Rightarrow
\langle N,\nabla_{\alpha'}\frac{\partial}{\partial x}\rangle=z'e^z\langle N,E_1\rangle=0\hspace*{.5cm}\mbox{along $\partial M$}.$$
Thus
$$\langle\nabla_{\alpha'}N,\nu\rangle=\langle\nabla_{\alpha'}N,e^{-z}\frac{\partial}{\partial x}\rangle=0\hspace*{.5cm}\mbox{along $\partial M$}.$$
This means that the geodesic curvature of $\Gamma$ in $M$ is zero, and so, $\Gamma$ is a line of curvature of $M$. The orthogonality between $M$ and $P$ is a consequence of $\langle N,E_1\rangle=0$.
\end{proof}

 It is known that the plane $P$ with the induced metric of Sol is isometric to a hyperbolic plane: it suffices the change $t=e^z$ and then $P\rightarrow \{(y,t)\in\r^2;t>0\}$ and the metric is $(dy^2+dt^2)/t^2$. With this change of variables the expression $\exp{\Big(\max_{p\in\Gamma}z(p)-\min_{p\in\Gamma}z(p)\Big)}$ converts into
$$\frac{\max_{p\in\Gamma}t(p)}{\min_{p\in\Gamma}t(p)}.$$
In the particular case that $\Gamma$ is a curve of constant (intrinsic) curvature $\kappa=c$, then $\Gamma$ is a circle of hyperbolic plane (remark that  $c$ must be greater than $1$, which assures that the curve is closed). The quantities that appear in (\ref{flux1}) are
$$\frac{\max_{p\in\Gamma}t(p)}{\min_{t\in\Gamma}t(p)}=\frac{c+1}{c-1},\ \ L(\Gamma)=\frac{2\pi}{\sqrt{c^2-1}},\ \ A(D)=
\frac{2\pi(c-\sqrt{c^2-1})}{\sqrt{c^2-1}}.$$
See \cite{br}. Thus

\begin{corollary}\label{c-cir}  Let $\Gamma$ be a circle of curvature $c$  included in the plane $P$. If $M$ is a compact  surface spanning $\Gamma$ and with constant mean curvature $H$, then
\begin{equation}\label{flux11}
|H|\leq \frac{c+1}{c-1}\frac{c+\sqrt{c^2-1}}{2}.
\end{equation}
\end{corollary}

 If  $\Gamma$ is a circle of curvature $c$ and we have equality in (\ref{flux1}) we obtain that the normal curvature of $\Gamma$ is
$c$. However, from (\ref{flux11}) we do not conclude that $\Gamma$ is a set of umbilical points, since
equation $H=c$ has no solution in the interval $(1,\infty)$. This contrasts to the Euclidean case, where if the boundary curve is a circle of  curvature $c$, the flux formula says $|H|\leq c$ and the equality occurs if $M$ is an umbilical surface, that is, $M$ is a hemisphere (\cite{be}).

\begin{theorem}\label{t-ere} Let $\Gamma$ be a closed embedded curve included in the plane $R$. If $M$ is a compact surface spanning $\Gamma$ and with constant mean curvature $H$, then
\begin{equation}\label{flux2}
|H|\leq \frac{\int_{\partial M}\sqrt{1+x^2+y^2}\ ds}{2 A(D)}.
\end{equation}
In particular, if $\Gamma$ is a circle of curvature $c$ we have $|H|\leq \sqrt{c^2+1}$.
\end{theorem}

\begin{proof}
Now we choose as  Killing vector field  $X= -x\frac{\partial}{\partial x}+y\frac{\partial}{\partial y}+\frac{\partial}{\partial z}$. The unit normal vector to $D$ is $\eta=E_3$ and $\langle \eta,X\rangle=1$. On the other hand,  $|\langle\nu,X\rangle|\leq |X|=\sqrt{1+x^2+y^2}$ .

\end{proof}

In particular, the inequalities (\ref{flux1}) and (\ref{flux2}) answer to the question posed in Introduction about
the possible values of (constant) mean curvatures of surfaces spanning a given curve: these values are not arbitrary, since the right-hand sides in both inequalities do not depend on the surface $M$ but the boundary curve $\Gamma$. In fact, we have upper bounds of $H$, which only depend  on the
geometry of the boundary curve $\Gamma$.

%%%%%%%%%%%%%%%%%%%%%%%%%%%%%%%%%%%%%%%%%%%%%%%%%%%%%%%%%%%%%%%%%%%%%%%%%%%%%%%%%%%%%%%%%
\section{The Alexandrov reflection method in Sol}\label{s-emb}
%%%%%%%%%%%%%%%%%%%%%%%%%%%%%%%%%%%%%%%%%%%%%%%%%%%%%%%%%%%%%%%%%%%%%%%%%%

In the theory of surfaces with constant mean curvature, the maximum principle plays an important role.
The maximum principle which says that if $M_1$ and $M_2$ are two surfaces tangent at some point with the mean curvature vectors oriented in the same direction and having the same constant mean curvature, then if $M_1$ lies to one side of $M_2$, the surface $M_1$ must coincide with $M_2$ in an open set.

As a consequence, if $M$ is a minimal surface in Sol with boundary $\Gamma$ included in the plane $P$,  the maximum principle asserts that $M\subset P$: it suffices to compare the surface $M$ with a plane $P_t$ at the highest and lowest points of $M$. For this reason, minimal surfaces of Sol with boundary in $P$ are  domains of $P$. Similar result is obtained  if $\Gamma$ is included in $R$.  The same occurs in Euclidean space. However, if $\Gamma$ is not a planar curve, the result is very different in both ambient spaces. In Euclidean space, and comparing with any plane, we obtain that a minimal surface spanning $\Gamma$ is contained in the convex hull of $\Gamma$. In contrast, a minimal surface in Sol bounded by $\Gamma$ lies in the convex hull of $\Gamma$ formed \emph{only} by planes of the families ${\cal F}_i$, $1\leq i\leq 3$.

Other consequence of the maximum principle is  the Alexandrov reflection method, which   appears in the literature as a powerful technique in the study of symmetries of a CMC surface,
specially if the surface is embedded.  Using this technique,
 the very Alexandrov showed that round spheres are the only closed (compact and without boundary) CMC surfaces which are embedded in Euclidean space (\cite{al}).

The Alexandrov reflection method  consists into consider a foliation of the space by  totally geodesic surfaces
   and by a process of reflection   and comparison, together the maximum principle, one shows that there exists a symmetry of the surface. In space Sol this can  do by  reflections across the planes $P_t$ and $Q_t$. For example, a closed embedded CMC surface in Sol is topologically a sphere \cite{dm}.

   Before continuing,  we define a symmetric bigraph in Sol. A closed curve $\Gamma$ included in the plane $P$ is said
 a \emph{symmetric bigraph} (with respect to the $y$-direction) if $\Gamma$ is symmetric with respect to the reflection across
 the plane $Q$ and each one of the two components of $\Gamma$ divided by  $Q$   is a graph on $l:=P\cap Q$. A direct consequence of the Alexandrov reflection method is the following result:

\begin{theorem}\label{t-alex} Let $M$ be a compact embedded CMC surface in Sol with  boundary  included in  $P$. Assume
\begin{enumerate}
\item $M$ lies in one side of $P$.
\item The boundary $\partial M$ is a symmetric bigraph.
\end{enumerate}
Then $M$ is invariant by the reflections across to the plane $Q$. Moreover $Q$ divides $M$ in two symmetric graphs over some domain of the plane $Q$.
\end{theorem}

As conclusion the surface is a symmetric bigraph with respect to the $y$-direction.

\begin{remark}
When we say that $M$ lies in one side of $P$ we mean that $M\subset P^+=\{(x,y,z);x\geq 0\}$ and $M-\partial M\subset
\{(x,y,z);x>0\}$. In fact, it is impossible that $M\subset P^+$ and $M\cap P$ contains interior points of $M$: by comparison $M$ with $P$ at these points, we would have $H\equiv 0$, and $M\subset P$.
\end{remark}

A similar result is obtained if the boundary is included in the horizontal plane $R$. The hypothesis are similar and we only have to give the concept of symmetric bigraph since there exist two types of such curves in the plane $R$. So, we say that a closed curve $\Gamma$ included in the plane $R$ is  a symmetric bigraph with respect to the $x$-direction (resp. the $y$-direction) if $\Gamma$ is symmetric with respect to the reflection across
 the plane $P$ (resp. $Q$)  and each one of the two components of $\Gamma$ divided by  $P$ (resp. $Q$)   is a graph on $P\cap R$ (resp. $Q\cap R$).

Due to Theorem \ref{t-alex}, one asks by those conditions that ensure that the surface lies in one side of the plane containing the boundary. The next result holds for surfaces with non-constant mean curvature (see \cite{ko} for the Euclidean version) and it uses the fact that the foliations ${\cal F}_i$ of the ambient space are minimal surfaces.

\begin{theorem} Let  $M$ be a compact  embedded surface in Sol whose boundary $\Gamma$ lies in the plane $P$.
Denote by $D\subset P$ the bounded domain by $\Gamma$. Assume
\begin{enumerate}
\item The mean curvature function $H$ does not vanish.
\item The surface $M$ does not intersect $ext(D)$.
\end{enumerate}
Then $M$ lies in one of the half-spaces of $\r^3$ determined by $P$.
An analogous result is obtained by replacing $P$ by $R$.
\end{theorem}

\begin{proof} Without loss of generality, we assume that $M$ has points with positive $x$-coordinate. Then we show that $M\subset P^+$. We have to show that $M\cap D=\emptyset$. By contradiction, we
assume that $M\cap D\not=\emptyset$ and thus  $M$ has points in both sides of $P$.

Consider a sufficiently big Euclidean half-sphere $S$ included in $\r^3-P^+$, with $\partial S\subset P$ and such that $S$ together
the annulus $A\subset P$ bounded by $\partial S$ and $\Gamma$, and
the very surface $M$ forme a closed embedded  surface $M^*=S\cup
A\cup M$. The surface $M^*$ is not smooth along the curves $\partial A$, but this has no effect in the reasoning. Denote $W\subset\r^3$ the enclosed domain by $M^*$. We choose an orientation $N$ on $M$
 so the mean curvature $H$ is positive. We now study if $N$ points to $W$.

  Let $p$ be the highest point (with respect to $x$-direction) and let $x(p)$ be its $x$-coordinate. Comparing $M$ with the minimal surface $P_{x(p)}$, and using the maximum principle, the vector  $N(p)$ points down, and so, towards $W$.  On the other hand, in the lowest point $q\in M$ and because the vector $N(q)$  points to $W$, then
   $N(q)$ points down again.
 Let us place at $q$ the minimal surface $P_{x(q)}$. Then the maximum principle yields a contradiction since $P_{x(q)}$ lies in one side of $M$ around $q$, but  $P_{x(q)}$ is a minimal surface and the mean curvature of $M$ is positive. This contradiction concludes the proof of Theorem.
\end{proof}

The second result establishes conditions for a CMC surface to be a graph on $P$. We precise in this moment the notion of (Killing) graph in Sol. Given a domain $D\subset P$ and $\frac{\partial}{\partial x}$ the Killing vector field orthogonal to $P$, denote $\Psi:\r\times P\rightarrow Sol$ the flux generated by $\frac{\partial}{\partial x}$. Then the graph of a function $u$ on $D$ is the surface $\{\Psi(u(p)),p);p\in D\}$. As the flow lines of $\frac{\partial}{\partial x}$ are horizontal straight-lines orthogonal to $P$, a graph on $D$ coincides with the concept of graph from the Euclidean viewpoint.

\begin{theorem} Let $M$ be a compact  embedded CMC  surface bounded by a closed curved contained in the plane $P$. Denote by $D$ the bounded domain by $\partial M$ in $P$.
Assume that $M$ lies in one side of $P$ and that it is a graph on $D$ around  $\Gamma$. Then  $M$ is a graph on $D$.
\end{theorem}

\begin{proof} Suppose that $M$ lies in the half-space $P^+$. We use the Alexandrov method with reflections across the planes $P_t$ coming from infinity. Let us remark that $M$ together $D$ encloses a domain of $\r^3$.  If $t$ is sufficiently big, $P_t$ does not intersect $M$. Moving  $P_t$ towards $P$, that is $t\searrow 0$, we arrive until the first contact point with $M$ at $t=t_0$. Next, we take $P_t$ for $t<t_0$ and let reflect the part of $M$ on the upper-side of $P_t$ until the next contact of $M$ with itself at
$t=t_1$. If this occurs for the time $t_1=0$, we have proved that $M$ is a graph on $P$ and this finishes the proof. On the contrary, if $t_1>0$ the hypothesis about the boundary asserts that the contact occurs between interior points of $M$. The maximum principle would yield that $P_{t_1}$ is a plane of symmetry of $M$,  with $D$ in one side of the plane $P_{t_1}$: contradiction.

\end{proof}

One can easily extend the result changing the hypothesis $M\subset P^+$ by the fact that $M$ does not intersect the half-cylinder $\{(t,y,z);(0,y,z)\in\Gamma, t<0\}$.

%%%%%%%%%%%%%%%%%%%%%%%%%%%%%%%%%%%%%%%%%%%%%%%%%%%%%%%%%%%%%%%%%%%%%%%%%%%
\section{Applications of the flux formula}\label{s-app}
%%%%%%%%%%%%%%%%%%%%%%%%%%%%%%%%%%%%%%%%%%%%%%%%%%%%%%%%%%%%%%%%%%%%%

In this section we combine the flux formula  together the maximum principle to obtain results about configurations of compact  embedded CMC surfaces of Sol.

\begin{theorem}\label{t-trans} Let $\Gamma$ be a  symmetric bigraph contained in $P$ and denote by $D\subset P$ the bounded domain by $\Gamma$. Let $M$ be a compact  embedded CMC surface with boundary $\Gamma$. If $M$
is transverse to $P$ along $\Gamma$ and $M\cap D=\emptyset$, then $M$  is
invariant by the reflections across $Q$.
\end{theorem}

This result follows ideas in \cite{bemr} for CMC surfaces in Euclidean space. However, and due to the
lack of isometries in the ambient space, the statement of our result is stronger that its Euclidean version.

\begin{proof}
 Without loss of
generality, we can assume that in a neighbourhood of $\Gamma$, $M$
is included in the half-space $P^+$.  Denote by $W\subset\r^3$ the bounded domain by $M\cup D$.
We claim that the intersection $M\cap ext(D)$ can not have two or more  curves homotopic in $ext(D)$ to
$\Gamma$. This is showed with  the Alexandrov reflection method  using reflection across the planes $Q_t$: the fact that the number of components is greater than one assures the existence of a contact interior point at a time $t_1>0$.  Then the plane   $Q_{t_1}$ would be a plane of symmetry of $M$ which is a contradiction with the fact that $\Gamma$ lies in one side of $Q_{t_1}$.

Therefore, $M\cap ext(D)$ has at most one component homotopic in $ext(D)$ to $\Gamma$. Assume that such component $C$ do exist. In order to use the flux formula, we  orient $M$ so  the mean curvature $H$ is positive. Then the corresponding Gauss map $N$  points into $W$. In particular, along $\Gamma$, the vector $N$ points towards $ext(D)$. Now we use the flux formula  with the Killing vector field $\frac{\partial}{\partial x}$. By the orientations chosen in Lemma \ref{l-flux}, we have that the unit vector $\eta$ orthogonal to $D$ is $\eta=-e^{-z}\frac{\partial}{\partial x}$ and so $\langle\eta,\frac{\partial}{\partial x}\rangle=-e^z <0$. Because the surface $M$ is contained in $P^+$ around $\partial M$,
we have  $\langle\nu,\frac{\partial}{\partial x}\rangle>0$ along $\partial M$. The flux formula (\ref{eq3}) gives
\begin{equation}\label{trans}
0<\int_{\partial M}\langle\nu,\frac{\partial}{\partial x}\rangle=2H\int_D\langle \eta,\frac{\partial}{\partial x}\rangle<0.
\end{equation}
This contradiction implies that $M\cap ext(D)$ has no components homotopic to $\Gamma$ in $ext(D)$.

We show that  $M$ is invariant with respect to $Q$, proving the result. This a direct consequence of the  Alexandrov reflection method with this planes coming from infinity. Suppose that the family $Q_t$ intersects $M$ for the first time at $t=t_0$. Continuing the movement of $Q_{t_0}$ by parallel translations doing
$t\searrow 0$, it would produce for some time $t_1<t_0$ a point of  contact of $M$ with the reflection of $M\cap (\cup_{t_1\leq t_0}Q_t\cap M)$ in $Q_{t_{1}}$. The fact that there are no components of $M\cap ext(D)$ homotopic to $\Gamma$  in $ext(D)$ together that $M\cap D=\emptyset$ assures that this contact does not occur between an interior point and  a boundary point of $M$. If this point is a smooth point of $M$,  the maximum principle  yields a plane of symmetry of $M$, in particular, a symmetry of $\Gamma$. Thus $t_1=0$ and this proves the result.
 If the contact occurs for a boundary point, the fact that $\Gamma$ is symmetric bigraph implies that $t_1=0$. Now, we repeat the reasoning but with the planes $Q_t$ coming from $t=-\infty$. Then we show again that we can arrive until the position $t=0$, showing in fact that $Q$ is a plane of symmetry of $M$.
\end{proof}

\begin{remark} In contrast to the Euclidean version, the surface $M\cap ext(D)$ could have components
 nulhomotopic in $ext(D)$. If this would be, the same proof shows that such components are symmetric bigraphs and their interiors are mutually disjoint.
 \end{remark}

\begin{remark} We are not able to extend Theorem \ref{t-trans} to the case that $\Gamma$ is included in the horizontal plane $R$. The  step in (\ref{trans}) does not work here. For this, we take as Killing vector field $X=-x\frac{\partial}{\partial x}+y\frac{\partial}{\partial y}+\frac{\partial}{\partial z}$. Assuming
that  $M$ lies in $R^+$ around $\partial M$,  $\eta=-E_3$ and so $\langle\eta,X\rangle=-1 <0$. But we can not control the sign of $\langle\nu,X\rangle$ since
$\langle\nu,X\rangle=-x\nu_1+y\nu_2+\nu_3$, where $\nu_i$ are the coordinates with respect to the basis
$\{\frac{\partial}{\partial x},\frac{\partial}{\partial y},\frac{\partial}{\partial z}\}$.  The fact that $M$ lies in $R^+$ means $\nu_3>0$. Even in the simplest case of $\Gamma$, that is, $\alpha$ is a circle,  $\alpha(s)=(\cos(s),\sin(s),0)$, we
have $\langle\nu,X\rangle=-\mu(s)\cos(2s)+\nu_3(s)$ for a certain function $\mu$, with $\mu(s)^2+\nu_3(s)^2=1$.
\end{remark}

Given a domain $D$ of $P$, denote $Cyl(D)$ the Killing cylinder determined by $D$, that is,
$Cyl(D)=\r\times D$.

\begin{theorem} Let $\Gamma$ be a Jordan curve in $P$ enclosing a domain $D$.  Assume that $M$ and $G$ are two  compact surfaces spanning $\Gamma$ with constant mean curvatures $H_1$ and $H_2$ respectively  and contained both in $Cyl(D)$. If $|H_1|=|H_2|$ and $G$ is a graph on $D$, then $M=G$ of $M=G^*$, the reflection of $G$ across to $P$.
\end{theorem}

\begin{proof} The case $H=0$ is trivial since the maximum principle yields that both surfaces are the very domain $D$. Thus, we assume $H_i\not=0$. Using the maximum principle again, it is not difficult to show that $G$ lies in one side of $P$. We assume then $G\subset P^+$ (and so, $G^*\subset P^-$). We consider the orientation on $G$ such that $H_2>0$, that is, the orientation that points down (with respect to the $x$-direction).

By combining the maximum principle together translations of $G$ and $G^*$ in the $x$-direction, one proves that either $M$ coincides with $G$ or $G^*$, showing the Theorem, or $M$ lies included in the bounded domain by $G\cup G^*$. Recall that $\partial M=\partial G=\partial G^*=\Gamma$. If we compare the inner conormal vectors
$\nu_M$ and $\nu_G$, the fact that $M$ is sandwiched by $G$ and $G^*$ writes as
$$\Big|\langle\nu_M,\frac{\partial}{\partial x}\rangle\Big|<\langle\nu_G,\frac{\partial}{\partial x}\rangle.$$
The flux formula (\ref{eq3}) for each surface gives
$$2|H_1|\Big|\int_D\langle \eta,\frac{\partial}{\partial x}\rangle\Big|=\Big|\int_{\partial M}\langle \nu_M,\frac{\partial}{\partial x}\rangle\Big|<\int_{\partial M}\langle \nu_G,\frac{\partial}{\partial x}\rangle
=2|H_1|\Big|\int_D\langle \eta,\frac{\partial}{\partial x}\rangle\Big|.$$
This contradiction shows the result.
\end{proof}

\begin{remark} Let $\Gamma$ be a closed curve in $P$. For the existence of graphs with constant mean curvature $H$ and boundary $\Gamma$ we can use the result established  in \cite{dhl}.  The Killing cylinder based in $\Gamma$ has mean curvature curvature $H_{cyl}=\theta'/2$, where  $\theta$
is the angle that appears in the parametrization by the arc-length of $\Gamma$ (see Appendix for local computations of the curvature). First we have to  assume that
$\inf_{Sol} Ric\geq-2\inf_{\Gamma} H_{cyl}^2$. In our case, the infimum of the Ricci tensor is $-2$ (\cite{dm}), the curvature of
$\alpha$ is $\kappa=\theta'+\cos\theta$. Assume that $\kappa>1$. Then
the condition writes as
$$-2\geq -2\inf_{\Gamma}\frac{\theta'^2}{4} \Rightarrow 4\leq \inf (\kappa-\cos\theta)^2.$$
For this it suffices that $4\leq (\kappa-1)^2$.
Then the condition for existence of graphs with constant mean curvature $H$ is that
$|H|\leq \inf_\Gamma H_{cyl}$, that is, $|H|\leq \inf(\kappa-\cos\theta)/2$. A sufficient condition is that
$|H|\leq (\kappa-1)/2$.
\end{remark}
%%%%%%%%%%%%%%%%%%%%%%%%%%%%%%%%%%%
\section{Appendix}
%%%%%%%%%%%%%%%%%%%%%%%%%%%%%%%%%%%%

The Riemannian connection ${\nabla}$ of Sol with respect to $\{E_1,E_2,E_3\}$  is
\begin{equation}\label{cov}\begin{array}{lll}
{\nabla}_{E_1} E_1=-E_3 & {\nabla}_{E_1}E_2=0&{\nabla}_{E_1}E_3=E_1\\
{\nabla}_{E_2} E_1=0 & {\nabla}_{E_2}E_2=E_3&{\nabla}_{E_2}E_3=-E_2\\
{\nabla}_{E_3} E_1=0 & {\nabla}_{E_3}E_2=0&{\nabla}_{E_3}E_3=0\\
\end{array}
\end{equation}

Let $\alpha$ be a curve contained in the plane $P$ and we compute its (intrinsic) curvature in $P$. Since $P$ is a totally geodesic surface, this curvature coincides with the curvature $\kappa$ of $\alpha$ as a submanifolds of  Sol.
Assume that the parametrization of $\alpha$ is $\alpha(s)=(0,y(s),z(s))$, $s\in I$,  where $s$ is the arc-length parameter. Then
$$e^{-z(s)}y'(s)=\cos\theta(s),\hspace*{1cm}z'(s)=\sin\theta(s),$$
where $\theta=\theta(s)$ is a certain smooth function. This allows to write  $\alpha'(s)=\cos\theta E_2+\sin\theta E_3$. Taking into account (\ref{cov}), we have
\begin{eqnarray*}
\nabla_{\alpha'}\alpha'&=& -\theta'\sin\theta E_2+\theta'\cos\theta E_3+\cos\theta\nabla_{\alpha'}E_2+
\sin\theta\nabla_{\alpha'}E_3\\
&=&-\theta'\sin\theta E_2+\theta'\cos\theta E_3+\cos\theta(\cosh\theta E_3)+\sin\theta(-\cosh\theta)\\
&=&(\theta'+\cos\theta)(-\sin\theta E_2+\cos\theta E_3.
\end{eqnarray*}
Thus $\kappa=|\nabla_{\alpha'}\alpha'|=\theta'+\cos\theta$.

We compute the mean curvature of a Killing cylinder $S$ based in a planar curve $\alpha$ contained in $P$.
We parametrize $S$ by $X(s,t)=(t,y(s),z(s))$, $s\in I\subset\r$, $t\in\r$,
where $\alpha$ is parametrized by the arc-length.  We have
\begin{eqnarray*}
& &e_1:=X_s=(0,y',z')=\cos\theta E_2+\sin\theta E_3.\\
& &e_2:=X_t=(1,0,0)=e^{z}E_1.
\end{eqnarray*}
We choose as Gauss map $N=-\sin\theta E_2+\cos\theta E_3$. We know that
$$H=\frac12\frac{Eg-2Ff+Ge}{EG-F^2},$$
with
$$\begin{array}{lll}
 E=\langle e_1,e_1\rangle, &F=\langle e_1,e_2\rangle, &G=\langle e_2,e_2\rangle.\\
e=\langle N,\nabla_{e_1}e_1\rangle, & f=\langle N,\nabla_{e_1}e_2\rangle, &g=\langle N,\nabla_{e_2}e_2\rangle.
\end{array}$$
In our case, the coefficients of the first fundamental form are
$$E=1,\ F=0,\ G=e^{2z},$$
and $EG-F^2=e^{2z}$. The values of $\nabla_{e_i}e_j$ are
\begin{eqnarray*}
\nabla_{e_1}e_1&=&(\theta'+\cos\theta)(-\sin\theta E_2+\cos\theta E_3).\\
\nabla_{e_1}e_2&=&\nabla_{e_2}e_1=\sin\theta e^{z}E_1.\\
\nabla_{e_2}e_2&=&-e^{2z}E_3.
\end{eqnarray*}
 Then
$$e=\theta'+\cos\theta,\hspace*{.5cm} f=0,\hspace*{.5cm} g=-e^{2z}\cos\theta.$$
As conclusion, $H= \theta'/2$.

%%%%%%%%%%%%%%%%%%%%%%%%%%%%%%%%%%%%%%


\begin{thebibliography}{99}
%%%%%%%%%%%%%%%%%%%%%%%%%%%%%%%%

\bibitem{al} Alexandrov, A. D.:
Uniqueness theorems for surfaces in the large I. Vestnik Leningrad
Univ. Math. 11, 5--17 (1956)

\bibitem{br} Bridgeman, M.: Average curvature of convex curves in $H^2$. Proc. Amer. Math. Soc. 126, 221--224 (1998).



\bibitem{be}  Brito, F., Earp, R.: Geometric configurations of constant mean curvature surfaces with planar boundary. An. Acad. Brasi. Ci. 63, 5--19 (1991)

\bibitem{bemr} Brito, F., Earp, R.,   Meeks III, W.,   Rosenberg, H.:
 Structure theorems for constant mean curvature surfaces bounded by a planar curve.
Indiana Univ. Math. J.  40, 333--343 (1991)

\bibitem{dm}  Daniel, B.,  Mira, P.:
Existence and uniqueness of constant mean curvature spheres in Sol$_3$. Preprint, arXiv: 0812.3059v2

\bibitem{dhl} Dajczer, M., Hinojosa, P.,  Lira, J. H. S.: Killing graphs with prescribed mean
curvature. Calc. Var. Partial Diff. Equations 33, 231–-248 (2008)

\bibitem{il}  Inoguchi, J.,  Lee, S.: A Weierstrass type representation for minimal surfaces in  Sol. Proc. Amer. Math. Soc. 146, 2209--2216 (2008)

\bibitem{ko}  Koiso, M.:
 Symmetry of hypersurfaces of constant mean curvature with symmetric boundary.
  Math. Z. 191, 567–-574 (1986)

\bibitem{kks}  Korevaar, N.,  Kusner, R.,  Solomon, B.: The structure of complete embedded surfaces with constant mean curvature.  J. Differential Geom. 30, 465--503 (1989)



\bibitem{lo}  L\'opez, R.: Invariant surfaces in homogeneosu space Sol with constant curvature.  arXiv preprint,
(2009)

\bibitem{sot} Souam, R.,  Toubiana, E.:
Totally umbilic surfaces in homogeneous 3-manifolds. Comm. Math. Helv. 84,  673--704 (2009)

\bibitem{th}  Thurton, W.:
Three-dimensional geometry and topology. Princenton Math. Ser. 35, Princenton Univ. Press, Princenton, NJ, (1997)

\bibitem{tr} Troyanov, M.: L'horizon de SOL. Exposition. Math. 16, 441--479 (1998)

\end{thebibliography}
\end{document}